\documentclass[12pt,a4paper]{amsart}

\usepackage{amsfonts, amsmath, amssymb, amsthm, amscd, hyperref}

\usepackage{anysize}
\usepackage{cancel}

\newtheorem{thm}{Theorem}[section]
\newtheorem*{thm*}{Theorem}
\newtheorem{lem}[thm]{Lemma}

\newtheorem{cor}[thm]{Corollary}

\theoremstyle{definition}
\newtheorem{defn}[thm]{Definition}
\newtheorem{ques}[thm]{Question}

\theoremstyle{remark}
\newtheorem{rem}[thm]{Remark}

\numberwithin{equation}{section}

\newcommand{\id}{{\rm id}}

\renewcommand{\int}{\mathop{\rm int}}

\renewcommand{\epsilon}{\varepsilon}

\begin{document}

\title{Estimating the higher symmetric topological complexity of spheres}

\author{Roman~Karasev}
\address{Roman Karasev, Dept. of Mathematics, Moscow Institute of Physics and Technology, Institutskiy per. 9, Dolgoprudny, Russia 141700}
\address{Roman Karasev, Laboratory of Discrete and Computational Geometry, Yaroslavl' State University, Sovetskaya st. 14, Yaroslavl', Russia 150000}
\email{r\_n\_karasev@mail.ru}
\urladdr{http://www.rkarasev.ru/en/}
\thanks{The research of R.N.~Karasev is supported by the Dynasty Foundation, the President's of Russian Federation grant MK-113.2010.1, the Russian Foundation for Basic Research grants 10-01-00096 and 10-01-00139, the Federal Program ``Scientific and scientific-pedagogical staff of innovative Russia'' 2009--2013, and the Russian government project 11.G34.31.0053.}

\author{Peter~Landweber}
\address{Peter Landweber, Department of Mathematics, Rutgers University, Piscataway, NJ 08854, USA} 
\email{landwebe@math.rutgers.edu}


\begin{abstract}
We study questions of the following type: Can one assign continuously and $\Sigma_m$-equivariantly to any $m$-tuple of distinct points on the sphere $S^n$ a multipath in $S^n$ spanning these points? A \emph{multipath} is a continuous map of the wedge of $m$ segments to the sphere. This question is connected with the \emph{higher symmetric topological complexity} of spheres, introduced and studied by I.~Basabe, J.~Gonz\'alez, Yu.~B.~Rudyak, and D.~Tamaki.  In all cases we can handle, the answer is negative.  Our arguments are in the spirit of the definition of the Hopf invariant of a map $f: S^{2n-1} \to S^n$ by means of the mapping cone and the cup product.
\end{abstract}

\maketitle

\section{Introduction}

Let us begin with some definitions following I.~Basabe, J.~Gonz\'alez, Yu.~B.~Rudyak, and D.~Tamaki~\cite{bgrt2010}. For $m \geq 2,$ denote by $J_m$ the space obtained by gluing $m$ segments $I_i = [0,1]_i$ $(i=1, \ldots, m)$ together at $t=0$. Denote the set of $m$ right ends (for $t=1$) of these segments by $\partial J_m$.

\begin{defn}
For a topological space $X$ let $P_m(X)$ be the space of all continuous maps $J_m\to X$ with the compact-open topology.
\end{defn}

\begin{defn}
Consider the fibration
$$
e_{m, X} : P_m (X) \to X^{\times m},
$$ 
which assigns to a map $f : J_m\to X$ the $m$-tuple of points $f(1_i)$ comprising $f(\partial J_m)$. This map is $\Sigma_m$-equivariant with respect to the natural actions of the permutation group $\Sigma_m$ on $P_m(X)$ (by right multiplication) and on $X^{\times m}$.
\end{defn}

The Schwarz genus (see~\cite{schw1966}) of the fibration $e_{m,X} : P_m(X)\to X^{\times m}$ is called the \emph{$m$-th topological complexity} $TC_m(X)$ of $X$ in~\cite{bgrt2010}.  (We always use the ``reduced'' Schwarz genus, so a fibration has a continuous global cross section if and only if its Schwarz genus is $0.$)  Making use of the actions of $\Sigma_m$ and the equivariance of the projection map, the equivariant Schwarz genus (see~\cite{schw1966}) of the fibration $e_{m,X} : P_m(X)\to X^{\times m}$ is called the \emph{$m$-th symmetric topological complexity} $TC_m^{\Sigma}(X)$ of $X$ in~\cite{bgrt2010}.

%RK Changed according to comments of PL 2011-11-03
%start
The study of topological complexity was begun by M.~Farber in~\cite{far2003}, and much of its development is summarized in Farber's book~\cite[Chapter~4]{far2008}. The number $TC_2(X)$ coincides with the Schwarz genus of the fibration $P_2(X) \to X \times X$.  In this case $P_2(X)$ agrees with the space of paths in $X$ because $J_2$ is just a segment.  A section of $e_{2,X}$ over an open subset of $X \times X$ may be viewed as a partial motion planning algorithm, which is a continuous assignment of paths joining certain pairs of points of $X$.  The minimal number minus one of open subsets of $X \times X$ which cover $X \times X$ and over which sections of $e_{2,X}$ exist is the reduced Schwarz genus of the fibration $e_{2,X}$; hence the minimal topological complexity of a motion planning algorithm coincides with $TC_2(X)$. If one requires a motion planning algorithm to be symmetric (the considered open subsets of $X \times X$ contain $(y,x)$ when they contain $(x,y)$, and the path assigned to a pair $(y,x)$ is the reverse of the path assigned to $(x,y)$), then the corresponding topological complexity becomes the symmetric topological complexity $TC^{\Sigma}_2(X)$.

The higher topological complexity $TC_m(X)$ is obtained from motion planning algorithms spanning $m$-tuples of points by a path; it does not really matter whether we span points by a path or by an image of $J_m$ (see~Yu.~Rudyak \cite{ru2010}).  The intuitive meaning of $TC^{\Sigma}_m(X)$ for motion planning algorithms is less clear, since in this case it is essential that  $J_m$ is acted on by the symmetric group, and any ``model'' chosen as a replacement for $J_m$ would also need an action by the symmetric group. 
%end

\begin{defn}
Define the configuration space $F_m(X)$ to be 
$$
F_m(X) = \{(x_1, \ldots, x_m)\in X^{\times m} : x_i \neq x_j \text{ for } i \neq j\}.
%\forall i\neq j\ x_i\neq x_j\}.
$$
\end{defn}

In~\cite[Theorem~4.12]{bgrt2010} Basabe, Gonz\'alez, Rudyak and Tamaki show that lower and upper  estimates of higher symmetric topological complexity involve the equivariant Schwarz genus of the restricted fibrations $\varepsilon_{m, X} := e_{m,X}|_{e_{m,X}^{-1}(F_m(X))}$; it is important that $\Sigma_m$ acts freely on $F_m(X)$ and its preimage under $e_{m,X}$. The first step is to distinguish between (reduced) equivariant Schwarz genus zero and positive equivariant Schwarz genus. This leads to the following question, raised and answered in some particular cases in~\cite{bgrt2010}, and studied further in this paper:

\begin{ques}
\label{mult-path}
Does the fibration $\varepsilon_{m, S^n}$ have a $\Sigma_m$-equivariant section?
\end{ques}

Here we are going to prove:

\begin{thm}
\label{mult-path-2mod4}
If $m\ge 3$ and $n \equiv 2 \pmod 4$, then $\varepsilon_{m, S^n}$ cannot have a $\Sigma_m$-equivariant section.
\end{thm}

\begin{thm}
\label{tri-path-0mod4}
If $n \equiv 0 \pmod 4, \, n>4,$ and $n/4$ is not a power of $3$, then $\varepsilon_{3, S^n}$ cannot have a $\Sigma_3$-equivariant section.
\end{thm}

% RK. Corollary for $3\mid m$ is removed.

%\begin{cor}
%\label{div3-path-0mod4}
%If $n \equiv 0 \pmod 4, \, n>4,$ $n/4$ is not a power of $3$, and $m$ is divisible by $3$, then $\varepsilon_{m, S^n}$ cannot have a $\Sigma_m$-equivariant section.
%\end{cor}

%RK I have rephrased the following sentence.

We state a more general result for all odd primes (including $3$) as well as their multiples:

\begin{thm}
\label{divp-path-0mod4}
If $m$ is divisible by an odd prime $p$ and $n=4n'p^s$ where $s\geq 0$, $p\nmid n'$, and $n'>\frac{p-1}{2}$, then $\varepsilon_{m, S^n}$ cannot have a $\Sigma_m$-equivariant section.
\end{thm}

\begin{rem}
\label{two-odd-divisors-rem}
Assuming that the fibration $\varepsilon_{m, S^n}$ has an equivariant section for $m$ divisible by two odd primes $p<q$ and $n=4n'p^sq^t$ with $n'$ coprime with $p$ and $q$, we obtain the inequalities: $n'p^s\le \frac{q-1}{2}$ and $n'q^t\le \frac{p-1}{2}$. From the latter inequality it follows that $t=0$, $n=4n'p^s$, and the former inequality now becomes $n\le 2(q-1)$. Hence for any $m$ with two distinct odd prime divisors the fibration $\varepsilon_{m, S^n}$ does not have an equivariant section for all $n$ apart from a finite number of possible exceptions.
\end{rem}

The proofs of the above theorems rely on actually proving nonexistence of a $D_{2m}$-equivariant section over the space of regular $m$-gons; see Section~\ref{mgons-sec} for the definitions. But the proofs of the following theorems require the whole configuration space:

\begin{thm}
\label{mult-path-0mod4}
If $m\ge 5$ is a prime and $n \equiv 0 \pmod 4$, then $\varepsilon_{m, S^n}$ cannot have a $\Sigma_m$-equivariant section.
\end{thm}

\begin{thm}
\label{even-path}
If $m\ge 4$ is even and $n\ge 2$, then $\varepsilon_{m, S^n}$ cannot have a $\Sigma_m$-equivariant section.
\end{thm}

%RK We cannot extend this theorem to multiples $m'$ of $m$ because for $m$ we need the full configuration space $F_m(S^n)$, not only $D_m(S^n)$!

\begin{rem}
In~\cite[Proposition~5.5]{bgrt2010} it is proved that $\varepsilon_{m, S^n}$ has no $\Sigma_m$-equivariant section for odd $n$ and any $m\ge 2$.  In the case of odd $n$ there is a simple cohomological obstruction based on the degree of a map of a sphere to itself, while in the above theorems the obstruction uses an analogue of the Hopf invariant. In addition, by using the mod 2 degree, it is proved in \cite{bgrt2010} prior to the proof of Proposition~5.5 that $\varepsilon_{2, S^n}$ has no $\Sigma_2$-equivariant section for all $n \geq 1.$
\end{rem}

\begin{rem}
The first open case of Question~\ref{mult-path} is therefore $m=3$, $n=4$; see the discussion of this case in the final section.  Nonetheless, the theorems proved here (and Remark~\ref{two-odd-divisors-rem}), together with the lack of any approach that seems likely to produce equivariant sections, provide evidence for the conjecture that there is no case with $m\geq 2$ in which $\varepsilon_{m ,S^n}$ has a $\Sigma_m$-equivariant section.
\end{rem} 

\begin{rem}
In~\cite[Corollary~5.4]{bgrt2010} it is shown that the Schwarz genus of $\varepsilon_{3, S^n}$ is at most $1$. Hence the above theorems imply that for $n$ not divisible by $4$ and $n$ not of the form $4\cdot 3^s$ with $s \geq 0$ the Schwarz genus of $\varepsilon_{3, S^n}$ is equal to $1$.
\end{rem}

The above theorems are negative results, but they lead to a positive consequence:

%RK Does the first sentence sound OK? 
%PSL IF you are referring to the first sentence of the corollary, note the slight changes I made.  

\begin{cor}
\label{coincidence}
Assume $m\ge 2$ and $n\ge 2$ are as in any of the above theorems, or $n$ is odd. Then for any $\Sigma_m$-invariant map $f : F_m(S^n)\to S^n$, i.e. a map constant on orbits, there exists a configuration $(x_1,\ldots, x_m)\in F_m(S^n)$ such that $f(x_1,\ldots, x_m)$ coincides with one of the points $x_1,\ldots, x_m$.
\end{cor}

\begin{rem}
From the proofs of Theorems~\ref{mult-path-2mod4}--\ref{divp-path-0mod4} and \cite[Proposition~5.5]{bgrt2010} it follows that in the corresponding cases of this corollary the configuration may be chosen to be a regular $m$-gon centered at the origin. Moreover, in these cases it is sufficient to assume that $f$ is defined only for such regular $m$-gons and is invariant with respect to the action of the dihedral group $D_{2m}$. 
%RK I have made the next sentence more precise.
If we rely on Theorem~\ref{mult-path-0mod4} then the configuration still may be chosen to be a regular $m$-gon, but the map $f$ has to be defined and $D_{2m}$-invariant on the whole configuration space. In Theorem~\ref{even-path} the domain is necessarily larger; see the details in Section~\ref{even-sec}.
\end{rem}

\begin{proof}
Assume the contrary and put 
$$
h(x_1,\ldots, x_m) = -f(x_1, \ldots, x_m).
$$ 
The point $h(x_1,\ldots, x_m)$ is \emph{not} antipodal to any of the points $x_1,\ldots, x_m$. So $h(x_1,\ldots, x_m)$ can be connected with every point $x_1,\ldots, x_m$ by a unique shortest path. Therefore we obtain a continuous $\Sigma_m$-equivariant section of $\varepsilon_{m, S^n}$, which cannot exist by the corresponding theorem of this paper or \cite[Proposition~5.5]{bgrt2010}.
%RK I have removed the list of theorems from the previous sentence.
\end{proof}

%RK I tried to outline the contents in the following paragraph.
%PSL This is fine.

The paper is organized as follows. In Section~\ref{reduction-sec} we reduce Question~\ref{mult-path} to a question about existence of certain $\Sigma_m$-equivariant maps from configuration spaces to a sphere. In Sections~\ref{mgons-sec} and \ref{tri-sec} we describe the spaces of regular $m$-gons. In Section~\ref{dihedral-coh-sec} we list the needed facts about the cohomology of dihedral groups. Then in Sections~\ref{mult-path-2mod4-sec}--\ref{even-sec} we prove the main results. In Section~\ref{m3-n4-sec} we discuss the first open case $m=3$, $n=4$ and explain why our approach fails in this case.

\bigskip
\textbf{Acknowledgments.}
We owe a major debt of gratitude to Fred Cohen for numerous discussions about these problems, and for the insights he has generously shared with the authors.  We are also grateful to Jesus Gonz\'alez for his interest and many helpful comments.

\section{Reduction of Question~\ref{mult-path}}
\label{reduction-sec}
Fix the numbers $m$ and $n$ and denote $\varepsilon_{m, S^n}$ simply by $\varepsilon$. A $\Sigma_m$-equivariant section of $\varepsilon : e^{-1}(F_m(S^n)) \to F_m(S^n)$ assigns to an $m$-tuple $(x_1,\ldots, x_m)\in F_m(S^n)$ a map 
$$
g_{(x_1,\ldots, x_m)} : J_m \to S^n
$$
so that the restriction $g_{(x_1,\ldots, x_m)}|_{\partial J_m}$ gives these $m$ points $x_1,\ldots, x_m$ in the prescribed order, and if we permute $(x_1,\ldots, x_m)$ the map $g_{(x_1,\ldots, x_m)}$ is right-multiplied by the corresponding permutation of branches of $J_m$.

Obviously, this gives an adjoint map $g : F_m(S^n)\times J_m \to S^n$. If we consider $J_m$ as $m$ segments $I_1,\ldots, I_m$ glued together, we see that $g$ gives $m$ maps
$$
g_i : F_m(S^n) \times I_i \to S^n.
$$
Define the map $\tilde g : F_m(S^n) \times I \to (S^n)^{\times m}$ (here $I=[0, 1]$) by 
$$
\tilde g((x_1,\ldots, x_m), t) = \left(g_1((x_1,\ldots, x_m), t), \ldots, g_m((x_1,\ldots, x_m), t)\right).
$$ 
The conditions on the section $g$ are equivalent to the following: $\tilde g(\cdot,t)$ is $\Sigma_m$-equivariant for all $t$,  
$\tilde g(\cdot, 1)$ coincides with the standard inclusion 
$$
\iota : F_m(S^n)\subseteq (S^n)^{\times m},
$$
and $\tilde g(\cdot, 0)$ assigns to any $m$-tuple $(x_1,\ldots, x_m)$ an $m$-tuple
$$
(g_1((x_1,\ldots, x_m), 0), \ldots, g_m((x_1,\ldots, x_m), 0))
$$ 
of \emph{equal} points of $S^n$ (because $I_1, \ldots, I_m$ are glued at $t=0$ to form $J_m$).

Putting $\tilde g(\cdot, 0) = (h(\cdot), \ldots, h(\cdot))$ we have proved:

\begin{lem}
\label{reduction1}
Question~\ref{mult-path} is equivalent to the following: Is there a $\Sigma_m$-equivariant homotopy between the natural inclusion $\iota :F_m(S^n)\to (S^n)^{\times m}$ and the map $\Delta\circ h$, where 
$$
\Delta : S^n \to (S^n)^{\times m}
$$
is the inclusion of the thin diagonal, and $h : F_m(S^n)\to S^n$ is some $\Sigma_m$-equivariant map? 
\end{lem}

\begin{rem}
Here and below we assume that the sphere $S^n$ has trivial $\Sigma_m$-action.
\end{rem}

In the following lemma the action of $\Sigma_{m-1}$ on $F_m(S^n)$ is the action by permutations of $m$-tuples $(x_1,x_2, \ldots,x_m)$ which fix $x_1.$

\begin{lem}
\label{reduction2}
The fibration $\varepsilon_{m, S^n}$ has a $\Sigma_m$-equivariant section over $F_m(S^n)$ if and only if there exists a $\Sigma_m$-equivariant map $h: F_m(S^n)\to S^n$, $\Sigma_{m-1}$-equivariantly homotopic to the composition $\pi_1\circ \iota$, where $\pi_1 : (S^n)^{\times m}\to S^n$ is the natural projection to the first factor.
\end{lem}

\begin{proof}
Assume that a $\Sigma_m$-equivariant section of $\varepsilon_{m, S^n}$ exists. By Lemma~\ref{reduction1} we have a $\Sigma_m$-equivariant map $h : F_m(S^n)\to S^n$ such that $\Delta\circ h$ is equivariantly homotopic to $\iota : F_m(S^n)\to (S^n)^{\times m}$. Then left-multiplying by the $\Sigma_{m-1}$-equivariant map $\pi_1$ we obtain
$$
\pi_1\circ \Delta\circ h \simeq_{\Sigma_{m-1}} \pi_1\circ \iota.
$$
It remains to note that $\pi_1\circ \Delta = \id_{S^n}$.

Now assume that there exists a $\Sigma_{m-1}$-equivariant map 
$$
g_1 : F_m(S^n)\times I \to S^n
$$
such that $g_1(\cdot, 0)$ is $\Sigma_m$-equivariant and $g_1((x_1,\ldots, x_m), 1) = x_1$. Put $g_k((x_1,\ldots, x_m), t) = g_1((x_k, x_1, \ldots, \cancel{x_k}, \ldots, x_m), t)$ for $k>1$ and 
$$
\tilde g((x_1, \ldots, x_m), t) = \left(g_1((x_1,\ldots, x_m), t), \ldots, g_m((x_1,\ldots, x_m), t)\right).
$$
The map $\tilde g$ is $\Sigma_m$-equivariant: 
\begin{multline*}
\tilde g((x_{\sigma(1)}, \ldots, x_{\sigma(m)}), t) = \left(g_1((x_{\sigma(1)}, \ldots, x_{\sigma(m)}), t), \ldots, g_m((x_{\sigma(1)}, \ldots, x_{\sigma(m)}), t)\right) =\\
= \left(g_1((x_{\sigma(1)}, \ldots, x_{\sigma(m)}), t), \ldots, g_1((x_{\sigma(m)}, x_{\sigma(1)}, \ldots, x_{\sigma(m-1)}), t)\right) =\\
= \left(g_{\sigma(1)}((x_1, \ldots, x_m), t), \ldots, g_{\sigma(m)}((x_1, \ldots, x_m), t)\right),
\end{multline*}
and obviously the map $\tilde g$ is a homotopy between $\iota$ (at $t=1$) and $\Delta\circ h$, where $h = g_1(\cdot, 0)$. So we apply Lemma~\ref{reduction1}.
\end{proof}

\section{Regular polygons in $F_m(S^n)$}
\label{mgons-sec}
We assume $m\ge 3$. Let us describe a subspace of $F_m(S^n)$.

\begin{defn}
We denote by $D_m(S^n)\subset F_m(S^n)$ the space of regular planar $m$-gons $x_1,\ldots, x_m$ centered at the origin. The vertices $x_1,\ldots, x_m$ lie on the unit sphere $S^n$ and are considered to be labeled in one of the two cyclic orders.
\end{defn}

%PSL One might wonder if this subset is a $D_{2m}$-equivariant deformation retract of $F_m(S^n)$ when $m>3$.  What do you know about this?
%RK In our paper with Volovikov we prove that for $m$ prime the (zero-based) Schwarz genus of $F_m(S^n)\to F_m(S^n)/\mathbb Z_m$ equals $(m-1)(n-1)+1$, while the Schwarz genus of $D_m(S^n)\to D_m(S^n)/\mathbb Z_m$ equals $2n-1$. So these spaces are essentially different for $m>3$.

This subset is not $\Sigma_m$-invariant unless $m=3$, but it is always invariant with respect to the action of the dihedral group $D_{2m}$ of order $2m,$ which we view as a subgroup of $\Sigma_m.$  In order to prove Theorems~\ref{mult-path-2mod4}, \ref{tri-path-0mod4} and \ref{divp-path-0mod4} we are going to use Lemma~\ref{reduction2} and replace the full space $F_m(S^n)$ by its subspace $D_m(S^n)$ and $\Sigma_m$-equivariance by $D_{2m}$-equivariance. We denote by $\mathbb Z_m$ the integers mod $m$, frequently viewed as the subgroup of $D_{2m}$ or the symmetric group $\Sigma_m$ consisting of the cyclic permutations.

The space $D_m(S^n)$ is easily seen to be homeomorphic to the Stiefel manifold $V_{n+1,2}.$ Namely, to $(x,y) \in V_{n+1,2}$ we assign $(x_1, \ldots, x_m) \in D_m(S^n),$ where $x_1 = x$ together with the further points $x_2,\ldots, x_m$ form the regular $m$-gon in the plane spanned by $x$ and $y$ so that the order of the vertices proceeds in the same sense as the rotation from $x$ to $y$. In this way, \emph{we shall identify these two spaces.}  We need the following information about the map $\pi_1 : D_m(S^n)\to S^n$ ($(x_1,\ldots, x_m)\mapsto x_1$), which is the restriction of the map $\pi_1$ from Lemma~\ref{reduction2} to the space of regular $m$-gons:

\begin{lem}
\label{st-bundle}
Let $n$ be even. The map $\pi_1 : D_m(S^n)\to S^n$ is the projection of the sphere bundle of the tangent bundle of $S^n$. Therefore $H^i(D_m(S^n)) = \mathbb Z$ in dimensions $i=0$ and $i=2n-1$, $H^n(D_m(S^n)) = \mathbb Z_2$, all other cohomology groups are zero, and the map $\pi_1^* : H^n(S^n)\to H^n(D_m(S^n))$ is a surjection.
\end{lem}

\begin{proof}
Since $D_m(S^n)$ has been identified with $V_{n+1, 2}$, the identification of $D_m(S^n)$ with the set of unit tangent vectors of $S^n$ is obvious. From the spectral sequence of this fibration with
$$
E_2(\pi_1) = H^*(S^n)\otimes H^*(S^{n-1})
$$
we see that the only nonzero differential is $d_n : E_n^{0, n-1}\to E_n^{n, 0}$, which is multiplication by $2$ (the Euler characteristic of $S^n$). This implies all the needed facts (which also follow from the Gysin sequence).
\end{proof}

Since $H^{2n-1}(D_m(S^n))=\mathbb Z$, we conclude that $D_m(S^n)$ is an orientable manifold. The following lemma shows that when $n$ is even $D_m(S^n)/D_{2m}$ is also an orientable manifold.

\begin{lem}
\label{orientationpres}
For $m \geq 3$ and even $n,$ the action of the dihedral group $D_{2m}$ on $D_m(S^n)$ preserves orientation.
\end{lem}

\begin{proof}
We first observe that the action of the cyclic subgroup $\mathbb Z_m$ preserves orientation.  Indeed, the action of $\mathbb Z_m$ extends to an action of the circle group, in which $e^{i \vartheta}$ maps a pair $(x,y)$ in the corresponding Stiefel manifold to $(\cos (\vartheta)\,x  + \sin (\vartheta)\, y, -\sin (\vartheta)\, x + \cos (\vartheta)\, y);$, so that $(x,y)$ is rotated through the angle $\vartheta$ in the plane they span; 
hence the action of $\mathbb Z_m$ preserves orientation.  It remains to show that one element of $D_{2m}$ not in its cyclic subgroup also preserves orientation, and for this we take the element which sends an element $(x_1, x_2, \ldots, x_m) \in D_m(S^n)$ to the element $(x_1, x_m, \ldots, x_2)$ listed in the reverse cyclic order.  Then each element $(x,y)$ of the corresponding Stiefel manifold is sent to $(x, -y).$ Hence we obtain an involution of the Stiefel manifold which for the bundle projection $(x,y) \mapsto x$ can be described as an automorphism which fixes the base $S^n$ and acts as the antipodal involution on each fiber.  The fibers are spheres of odd dimension, on which the antipodal map preserves orientation; hence the orientation of the Stiefel manifold is also preserved.  
\end{proof}

Denote by $\mathbb Z_2\subset D_{2m}$ the subgroup generated by the flip fixing $x_1$ (which appeared in the previous proof). The proof of Lemma~\ref{reduction2} yields the following reduction:

\begin{lem}
\label{reduction3}
If the fibration $\varepsilon_{m, S^n}$ has a $D_{2m}$-equivariant section over $D_m(S^n)$ then there exists a $D_{2m}$-equivariant map $h: D_m(S^n)\to S^n$, $\mathbb Z_2$-equivariantly homotopic to the composition $\pi_1\circ \iota$, where $\iota : D_m(S^n)\to (S^n)^{\times m}$ is the natural inclusion.
\end{lem}

\section{A note on triangles}
\label{tri-sec}
In the particular case $m=3$ we know even more. Define the function $E : F_3(S^n)\to \mathbb R$ by
$$
E(a,b,c) = \frac{1}{|a-b|} + \frac{1}{|b-c|} + \frac{1}{|c-a|}.
$$
This function is smooth and proper. A configuration $(a,b,c)\in F_3(S^n)$ is a critical point for $E$ if and only if $abc$ is a regular triangle centered at the origin. This can be proved in two steps: 
\begin{itemize}
\item 
If the affine $2$-flat $L$ spanned by $a,b,c$ does not contain the origin then we move $L$ towards the origin keeping $\triangle abc$ in it homothetic to itself. This motion decreases $E(a,b,c)$ with negative derivative.
\item 
When $0\in L$ express the side lengths of $\triangle abc$ by $2\sin\alpha$, $2\sin\beta$, and $2\sin\gamma$ for positive $\alpha,\beta,\gamma$ with $\alpha+\beta+\gamma = \pi$ so that 
$$
2E = \frac{1}{\sin\alpha} + \frac{1}{\sin\beta} + \frac{1}{\sin\gamma}.
$$ 
Differentiating twice implies that $E$ is strictly convex and therefore has a unique critical point (a minimum) at $\alpha=\beta=\gamma=\frac{\pi}{3}$.
\end{itemize}

Therefore the gradient flow of $E$ establishes a $\Sigma_3$-equivariant deformation retraction of $F_3(S^n)$ onto $D_3(S^n)$.

Another way to establish an equivariant homotopy equivalence in the case $n=2$ is to first identify $F_3(S^2)$ with the group of M\"obius transformations 
$\mathrm{PSL}(2, \mathbb C)$ of the Riemann sphere (it is classical that triples of points in $\mathbb CP^1=S^2$ admit a simply transitive action by $\mathrm{PSL}(2, \mathbb C)$), and to then notice that $\mathrm{SO}(3)$ is a maximal compact subgroup of 
$\mathrm{PSL}(2, \mathbb C).$
A non-equivariant homotopy equivalence $F_3(S^n)\simeq D_3(S^n)$ is established by E.R.~Fadell and S.Y.~Husseini in~\cite[Chapter III, Proposition~1.1, p.~32]{fh2001}.

Finally we obtain the following:

\begin{lem}
\label{reduction-tri}
The fibration $\varepsilon_{3, S^n}$ has a $\Sigma_3$-equivariant section over $F_3(S^n)$ if and only if there exists a $\Sigma_3$-equivariant map $h: D_3(S^n)\to S^n$, $\mathbb Z_2$-equivariantly homotopic to the composition $\pi_1\circ \iota$, where $\iota : D_3(S^n)\to (S^n)^{\times 3}$ is the natural inclusion.
\end{lem}

\section{The cohomology of dihedral groups}
\label{dihedral-coh-sec}
Before proving Theorems~\ref{mult-path-2mod4} and \ref{tri-path-0mod4} we need some facts about the cohomology of dihedral groups and the spaces $D_m(S^n)/D_{2m}$. The first lemma follows from the explicit description of the additive structure of $H^*(D_{2m})$ given by D.~Handel in the proofs of Theorems~5.2 and 5.3 in~\cite{han1993}:

\begin{lem}
\label{odd-torsion-group}
For $i>0$, if the cohomology group $H^i(D_{2m})$ fails to be $2$-primary then $i\equiv 0 \pmod 4$.  

More precisely, suppose that $p$ is an odd prime and that $p^r \|m$, i.e. $m = p^r k$ with $(p,k)=1.$  If $i>0$ and $i \equiv 0 \pmod 4,$ then the $p$-primary component of $H^i(D_{2m})$ is $\mathbb Z_{p^r},$  while if $j \not\equiv 0 \pmod 4$ then $H^j(D_{2m})$ has no $p$-torsion.    
\end{lem}

\begin{lem}
\label{odd-torsion-mgons}
If the cohomology group $H^n(D_m(S^n)/D_{2m})$ fails to be $2$-primary then $n \equiv 0 \pmod 4$. If the group $H^{n-1}(D_m(S^n)/D_{2m})$ fails to be $2$-primary then $n\equiv 1 \pmod 4$.
\end{lem}

\begin{proof}
Consider the Cartan--Leray spectral sequence~\cite{car1950} with 
$$
E_2^{p,q} = H^p(D_{2m}; H^q(D_m(S^n))),
$$
converging to $H^*(D_m(S^n)/D_{2m})$. By Lemma~\ref{orientationpres} the group $D_{2m}$ acts trivially on the cohomology of $D_m(S^n)$. From Lemma~\ref{st-bundle} it follows that the only nonzero rows of this spectral sequence occur for  $q=0,\,n,\,2n-1$. In the $n$-th row we have $H^p(D_{2m}; \mathbb Z_2)$, which has no odd torsion. Hence all the odd torsion in $H^n(D_m(S^n)/D_{2m})$ and $H^{n-1}(D_m(S^n)/D_{2m})$ comes from the group cohomology and we may apply Lemma~\ref{odd-torsion-group}.
\end{proof}

\begin{lem}
\label{torsion-evengons}
If $m$ is even and $n\equiv 2\pmod 4$, then the group $H^n(D_m(S^n)/D_{2m})$ is a $4$-torsion group. In the particular case $m=4$ the group $H^n(D_4(S^n)/D_8)$ is a $2$-torsion group.
\end{lem}  

Here an abelian group is called a \emph{$k$-torsion group} for a positive integer $k$ if each of its elements $a$ satisfies $ka=0.$
%Fred Cohen was not happy with this terminology.  His preference is to say that these cohomology groups are 2-torsion groups with exponents at most 4 and 2, respectively.  But I think our choice is fine.

\begin{proof}
It follows from Lemma~\ref{odd-torsion-group} that the group cohomology $H^n(D_{2m})$ is a $2$-torsion group. Consider again the spectral sequence with 
$$
E_2^{p,q} = H^p(D_{2m}; H^q(D_m(S^n)))
$$ 
with blank rows in the range $1\le q\le n-1$. If $d_{n+1}$ is nonzero on 
$$
E_{n+1}^{0, n} = H^0(D_{2m}; \mathbb Z_2) = \mathbb Z_2
$$
then $H^n(D_m(S^n)/D_{2m}) = H^n(D_{2m})$ and the proof is complete.

Otherwise we obtain the exact sequence
$$
0 \to H^n(D_{2m}) \to H^n(D_m(S^n)/D_{2m}) \to H^n(D_m(S^n))  \to 0
$$
and note that $H^n(D_m(S^n)/D_{2m})$ is a $4$-torsion group since $H^n(D_m(S^n)) = \mathbb Z_2.$ 

If $m=4$ we note that $D_4(S^n)/D_8 = V_{n+1,2}/D_8 \simeq F_2(\mathbb RP^n)/\Sigma_2 = B_2(\mathbb RP^n)$. The latter space is studied by C.~Dom\'{\i}nguez, J.~Gonz\'alez and the second author in~\cite{dgl2011}, and~\cite[Theorem~2.2]{dgl2011} implies that $H^n(B_2(\mathbb RP^n))$ is a $2$-torsion group when $n \equiv 2 \pmod 4$.
\end{proof}
%PSL But if $n \equiv 0 \pmod 4$, then this cohomology group contains a copy of $\mathbb Z_4$, so is not a $2$-torsion group.

%RK We may also state some facts about $H^*(D_m(S^n)/\mathbb Z_2)$ in view of Lemma~\ref{reduction3}, but it does not seem to help.

\section{Proof of Theorem~\ref{mult-path-2mod4}}
\label{mult-path-2mod4-sec}
We are going to use reasoning in the spirit of the definition of the Hopf invariant of a map $f: S^{2n-1} \to S^n$ by means of the mapping cone and the cup product.   

Using Lemma~\ref{reduction3}, we assume that a continuous map $h : D_m(S^n)\to S^n$ is $D_{2m}$-equivariant and $\mathbb Z_2$-equivariantly homotopic to the map  $\pi_1$ described in Lemma~\ref{st-bundle}. 

There results a map $h': D_m(S^n)/D_{2m}\to S^n$ so that $h$ admits the factorization $h=h'\circ q,$ where $q : D_m(S^n) \to D_m(S^n)/D_{2m}$ is the orbit projection.  Consider the mapping cones $X = S^n \cup_h CD_m(S^n)$ and $X' = S^n\cup_{h'} C(D_m(S^n)/D_{2m})$ and the corresponding commutative diagram:
$$
\begin{CD}
D_m(S^n) @>{h}>> S^n @>{\lambda}>> X @>>> S(D_m(S^n))\\
@VVV @| @V{\pi}VV @VVV\\
D_m(S^n)/D_{2m} @>{h'}>>S^n @>{\mu}>> X' @>{\nu}>> S(D_m(S^n)/D_{2m}).
\end{CD}
$$
There results a commutative diagram of long exact sequences of cohomology groups (note that $H^n(S(D_m(S^n))) \cong H^{n-1}(D_m(S^n))=0$):
\begin{equation}
\label{long-exact}
\begin{CD}
0 @>>> H^n(X) @>{\lambda^*}>> H^n(S^n) @>{h^*}>> H^n(D_m(S^n))\\
@AAA @A{\pi^*}AA @| @AAA \\
H^{n-1}(D_m(S^n)/D_{2m}) @>{\nu^*}>> H^n(X') @>{\mu^*}>> H^n(S^n) @>{h'^*}>> H^n(D_m(S^n)/D_{2m}).
\end{CD}
\end{equation}

The map $h$ is non-equivariantly homotopic to $\pi_1$. From Lemma~\ref{st-bundle} it follows that the space $X$ is homotopy equivalent to the Thom space of the tangent bundle of $S^n$, so its cohomology is isomorphic to $\mathbb Z$ in dimensions $0, n, 2n$ and zero in other dimensions. Denote generators by $x\in H^n(X)$ and $y\in H^{2n}(X)$. From the definition of the Euler class it follows that $x\smallsmile x = 2y$ for a suitable choice of the generator $y.$

Denote by $\alpha$ a generator of $H^n(S^n)$. By Lemma~\ref{st-bundle} in the upper row of (\ref{long-exact}) the map $h^*$ is a surjection and we have $\lambda^*(x) = \pm 2\alpha$.

Let us describe the lower row of (\ref{long-exact}). By Lemma~\ref{odd-torsion-mgons} and the hypothesis that $n \equiv 2 \pmod{4},$ the groups $H^n(D_m(S^n)/D_{2m})$ and $H^{n-1}(D_m(S^n)/D_{2m})$ are $2$-primary. Hence the group $H^n(X')/\nu^*(H^{n-1}(D_m(S^n)/D_{2m}))$ is isomorphic to $\mathbb Z$; denote by $x'$ an element of $H^n(X')$ which represents a generator of this infinite cyclic quotient group. It is clear that $\mu^*(x') = \pm s\alpha$ where $s$ is a power of $2.$ Since $\lambda^*(x) = \pm 2\alpha$ and $\mu^*(x')=\pm s\alpha$ it follows that $\pi^*(x') = \pm tx$, where $t$ is a power of $2,$ possibly equal to $1.$  (Observe that $s = 2t.$)  From the diagram (another part of the map between the long exact sequences)
$$
\begin{CD}
0 @>{h^*}>> H^{2n-1}(D_m(S^n)) @>>> H^{2n}(X) @>{\lambda^*}>> 0\\
@. @AAA @A{\pi^*}AA @. \\
0 @>{h'^*}>>H^{2n-1}(D_m(S^n)/D_{2m}) @>{\nu^*}>> H^{2n}(X') @>{\mu^*}>> 0
\end{CD}
$$
we conclude that the group $H^{2n}(X')$ is isomorphic to $\mathbb Z$ with a generator $y'$ such that $\pi^*(y')= 2my$, since the manifolds $D_m(S^n)$ and $D_m(S^n)/D_{2m}$ are orientable by Lemma~\ref{orientationpres}.

For some integer $c$ we have $x'\smallsmile x' = cy'$; this equality does not depend on adding an element of $\nu^*(H^{n-1}(D_m(S^n)/D_{2m}))$ to $x'$. Then applying $\pi^*$ we obtain $t^2 x\smallsmile x = 2mcy$, and therefore $t^2 = mc$. From the latter equation it is clear that $m$ must be a power of $2$. In this theorem $m\ge 3$, so in particular $m$ is even. Now we apply Lemma~\ref{torsion-evengons} to conclude that $s\le 4$, so $t=s/2\le 2$ and the equality $t^2 = mc$ is possible only for $m=4$ and $t=2.$  But applying Lemma~\ref{torsion-evengons} again for $m=4$ we obtain that $s\le 2$, and therefore $t=1$, a contradiction.  This completes the proof of Theorem 1.5.

\section{Proof of Theorem~\ref{tri-path-0mod4}}
\label{tri-path-0mod4-sec}

We repeat the above reasoning literally and only have to show that the map $\mu^*$ in (\ref{long-exact}) is multiplication by a power of $2$ on generators. In order to show this we have to prove that the image of the map $h'^* : H^n(S^n) \to H^n(D_3(S^n)/\Sigma_3)$ lies in the $2$-primary component of $H^n(D_3(S^n)/\Sigma_3)$.

Besides the $2$-primary component of $H^n(D_3(S^n)/\Sigma_3)$ there exists a $3$-primary component, isomorphic to $\mathbb Z_3$ (see Lemma~\ref{odd-torsion-group}). So we pass to modulo $3$ cohomology and to the cyclic subgroup $\mathbb Z_3\subset \Sigma_3$. Assume the map $h$ factors through $h'' : D_m(S^n)/\mathbb Z_3\to S^n$. Now we have to prove that
$$
h''^* : H^n(S^n; \mathbb F_3) \to H^n(D_3(S^n)/\mathbb Z_3; \mathbb F_3)
$$
is the zero map. The space $D_3(S^n)$ is a mod $3$ cohomology sphere having dimension $2n-1$, so the cohomology ring $H^*(D_3(S^n)/{\mathbb Z_3}; \mathbb F_3)$ is a truncation of the group cohomology, i.e. it has generators $u, w$ with $\dim u = 1$, $\dim w = 2$, satisfying the relations ($\beta$ denotes the Bockstein homomorphism)
$$
u^2 = 0,\quad \beta(u)=w, \quad w^n = 0.
$$

Note that a generator $\alpha\in H^n(S^n; \mathbb F_3)$ is annihilated by any cohomology operation of positive degree. But the Steenrod reduced powers of $w^{n/2}\in H^n(D_3(S^n)/\mathbb Z_3; \mathbb F_3)$ satisfy
$$
\mathcal P^i(w^{n/2}) = \binom{n/2}{i} w^{n/2+2i}.
$$
So if $\alpha$ maps to $c w^{n/2}$ with $c \in \mathbb{F}_3^*$, then $\binom{n/2}{i} \equiv 0 \pmod{3}$ for $i=1,\ldots,n/4 -1$.  Equivalently, putting $N = n/4$, $\binom{2N}{i} \equiv 0 \pmod{3}$ for $i=1,\ldots, N-1$.  By hypothesis $N$ does not have the form $3^s$ with $s \geq 0$, so a contradiction follows from:  
%Because of the relation $w^n=0$ we have to consider the binomial coefficients $\binom{n/2}{i}$ for $i=1,\ldots, n/4-1$. Equivalently, we put $N=n/4$ and consider the binomial coefficients $\binom{2N}{i}$, noting the following:

\begin{lem} 
For a positive integer $N$, all binomial coefficients $\binom{2N}{i}$ vanish mod $3$ for $0<i<N$ if and only if $N$ has the form $3^s$ for some $s\geq 0$.
\end{lem}

\begin{proof}
Let $t$ be an indeterminate. Then the vanishing of the binomial coefficients in the statement of the lemma is equivalent to requiring that
\begin{equation}
\label{binom-eq}
(1+t)^{2N} \equiv 1+\binom{2N}{N} t^N + t^{2N} \pmod{3}.
\end{equation}
In case $N$ is a power of $3$, possibly equal to $1$, this last condition is easily verified (and one learns that the ``middle'' binomial coefficient is congruent to $2$ mod $3$). Of course, it is the converse that is really needed here. Suppose $N$ is not a power of $3$. Then the formula of Lucas for binomial coefficients modulo a prime implies that there is a smallest integer $i,\, 0<i<N$, such that $\binom{N}{i} \not\equiv 0 \pmod{3}$. Then expand $(1+t)^{2N} =((1+t)^N)^2$ to find that there is a nonzero term $2 \binom{N}{i}  t^i$ appearing in the expansion of $(1+t)^{2N}$ mod $3$, which violates the condition (\ref{binom-eq}) at the start of the proof.
\end{proof}

\section{Proof of Theorem~\ref{divp-path-0mod4}}
\label{divp-path-0mod4-sec}
%Now we deduce Corollary~\ref{div3-path-0mod4}. We may uniquely extend any triangle $(a,b,c)\in D_3(S^n)$ to a regular $m$-gon $(a, \ldots, b, \ldots, c,\ldots)\in D_m(S^n)$, where ``$\ldots$'' denote $(m/3-1)$-tuples of points uniformly placed on geodesics $ab$, $bc$, and $ca$ on $S^n$. In this way, we obtain a $\Sigma_3$-equivariant map $\phi : D_3(S^n)\to D_m(S^n)$, which is in fact a homeomorphism. If $h : D_m(S^n) \to S^n$ satisfies the hypothesis of Lemma~\ref{reduction3} then so does $\phi\circ h$. Since the existence of $\phi\circ h$ is ruled out, the map $h$ cannot exist.

%RK Below if the unified proof of Theorem~\ref{divp-path-0mod4} and Corollary~\ref{div3-path-0mod4}.

First we note that both spaces $D_p(S^n)$ and $D_m(S^n)$ are identified with the same $V_{n+1,2}$. The corresponding homeomorphism $\phi: D_p(S^n)\to D_m(S^n)$ is obtained by adding $m/p-1$ further vertices uniformly placed on every geodesic arc $[x_ix_{i+1}]$ between successive (mod $p$) vertices of a $p$-gon. It is clear that $\phi$ is $D_{2p}$-equivariant. If $h : D_m(S^n) \to S^n$ satisfies the hypothesis of Lemma~\ref{reduction3} then so does $\phi\circ h : D_p(S^n)\to S^n$. So we have reduced the problem to the case of regular $p$-gons.

For odd primes $p$ the cohomology $H^*(\mathbb Z_p; \mathbb F_p)$ has structure similar to $H^*(\mathbb Z_3; \mathbb F_3)$ considered above, with a generator $w\in H^2(\mathbb Z_p; \mathbb F_p)$. The space $D_p(S^n)$ is a $(2n-1)$-dimensional mod $p$ homology sphere. As above, it suffices to find a cohomology operation on the cohomology $H^*(D_p(S^n)/\mathbb Z_p; \mathbb F_p)$ which acts nontrivially on $w^{n/2}$. 

We examine the mod $p$ Steenrod reduced powers, which satisfy
$$
\mathcal P^i(w^{n/2}) = \binom{n/2}{i} w^{n/2+(p-1)i}.
$$ 
For an indeterminate $t$, notice the following congruence of polynomials mod $p$:
$$
(1+t)^{n/2} = (1+t)^{2n'p^s} \equiv (1 + t^{p^s})^{2n'} = 1 + 2n't^{p^s} +\ \text{higher powers of}\ t.
$$
Assuming that all the Steenrod powers $\mathcal P^i(w^{n/2})$ with $n/2+(p-1)i < n$ vanish, we obtain the following (equivalent) inequalities:
$$
2n'p^s + (p-1)p^s\ge n = 4n'p^s\,\Rightarrow\, (p-1)\ge 2n'\,\Rightarrow\, n'\le \frac{p-1}{2},
$$
which contradicts the hypothesis of the theorem.

\section{Proof of Theorem~\ref{mult-path-0mod4}}
\label{mult-path-0mod4-sec}

In this theorem we use the squaring relation $(w^{n/2})^2=w^n$ instead of the Steenrod operations. In the mod $m$ cohomology of $D_m(S^n)/\mathbb Z_m$ we have $w^n=0$. (Recall that $m$ is assumed to be a prime, $m\geq 5$.)  So we have to consider the whole configuration space and use its cohomology.
 
We need to prove that the map $h''^* : H^n(S^n; \mathbb F_m)\to H^n(D_m(S^n)/\mathbb Z_m; \mathbb F_m)$ is trivial. The map $h''$ is a composition $h'' = \iota\circ h_0''$, where $\iota: D_m(S^n)/\mathbb Z_m\to F_m(S^n)/\mathbb Z_m$ is the natural inclusion and $h_0'' : F_m(S^n)/\mathbb Z_m \to S^n$ comes from the original map $h_0 : F_m(S^n)\to S^n$ satisfying the hypothesis of the theorem and Lemma~\ref{reduction1}.

Again let $w\in H^2(\mathbb Z_m; \mathbb F_m)$ be a generator. We need the following facts about the cohomology of $F_m(\mathbb R^n)/\mathbb Z_m$:

\begin{lem}
\label{FmRn}
Let $m\ge 5$ be a prime. The cohomology $H^n(F_m(\mathbb R^n)/\mathbb Z_m; \mathbb F_m)$ is generated by the power $w^{n/2}$, and its square $w^n$ is nonzero.
\end{lem}

%RK I have added the remark explaining the case $m=3$.

\begin{rem}
\label{F3Rn-rem}
For $m=3$ the cohomology $H^n(F_3(\mathbb R^n)/\mathbb Z_3; \mathbb F_3)$ is also generated by the power $w^{n/2}$, but $w^n=0$ because the group $H^{2n}(F_3(\mathbb R^n)/\mathbb Z_3;\mathbb F_3)$ vanishes. Moreover, the group $H^{2n}(F_3(\mathbb S^n)/\mathbb Z_3;\mathbb F_3)=H^{2n}(D_3(\mathbb S^n)/\mathbb Z_3;\mathbb F_3)$ also vanishes and the approach with squaring is not applicable to the case $m=3$.
\end{rem}

\begin{proof}
%RK I have commented the reference to my paper because Fred Cohen told it was known before.
%We mostly follow the proof of~\cite[Lemma~5]{kar2009}. 
The cohomology $H^*(F_m(\mathbb R^n); \mathbb F_m)$ is nonzero only in dimensions $0,n-1,2(n-1),\ldots,(m-1)(n-1)$; and it is composed of free $\mathbb F_m[\mathbb Z_m]$-modules except for dimensions $0$ and $(m-1)(n-1)$, see F.~Cohen and L.~Taylor ~\cite[Theorem~3.4, Corollaries~3.5 and 3.6]{ct1991}. Therefore the only nonzero entries in the $E_2$-term of the Cartan--Leray spectral sequence for $F_m(\mathbb R^n)$ are $H^*(\mathbb Z_m; \mathbb F_m)$ in the zeroth row, something nonzero in dimensions $i(n-1)$ ($0\le i\le m-1$) in the zeroth column, and something nonzero in the $(m-1)(n-1)$-th row. Because of the $H^*(\mathbb Z_m; \mathbb F_m)$-module structure of this spectral sequence the only nonzero differential in this spectral sequence can act from the $(m-1)(n-1)$-th row to the zeroth row, thus preserving $w^{n/2}$ and $w^n$ of dimensions $n, 2n \le (m-1)(n-1)$ (we use $m\ge 5$ and $n\ge 4$ here!) in the term $E_\infty$.
\end{proof}

Evidently, there is a $\Sigma_m$-equivariant inclusion $F_m(\mathbb R^n)\to F_m(S^n)$ and $w^n$ is also nonzero in $H^{2n}(F_m(S^n)/\mathbb Z_m; \mathbb F_m)$. Thus $w^{n/2}$ cannot be a multiple of $h_0''^*(\alpha)$ because $\alpha^2=0$. But this does not exclude the possibility $h''^*(\alpha)\neq 0$ because the group $H^n(F_m(S^n))/\mathbb Z_m; \mathbb F_m)$ may not be generated by $w^{n/2}$.

The inclusion $F_m(\mathbb R^n)\to F_m(S^n)$ is obtained by choosing a base point $x_0\in S^n$ and identifying $\mathbb R^n$ with $S^n\setminus\{x_0\}$. Let us describe the difference set $F_m(S^n)\setminus F_m(\mathbb R^n)$. This difference corresponds to configurations $(x_1,\ldots, x_m)$ such that for one index $i$ we have $x_i=x_0$. The other points $\{x_j\}_{j\neq i}$ form a configuration in $F_{m-1}(\mathbb R^d)$. Hence $F_m(S^n)\setminus F_m(\mathbb R^n)$ is a union of $m$ smooth manifolds $X_1\cup\dots\cup X_m$; each $X_i$ is diffeomorphic to $F_{m-1}(\mathbb R^n)$ and they are permuted freely by $\mathbb Z_m$. Moreover we note that $X_1$ is the preimage of zero under the natural projection map $\pi_1 : F_m(S^n)\to S^n$, so it has a normal framing in $F_m(S^n)$. It follows now that the pair 
$$
(F_m(S^n)/\mathbb Z_m, F_m(\mathbb R^n)/\mathbb Z_m) = (F_m(S^n), F_m(S^n)\setminus X_1)
$$ 
is homotopy equivalent to $(D^n\times X_1, S^n\times X_1)$, where $(D^n, S^n)$ is the $n$-ball and its boundary. Hence the cohomology $H^i(F_m(S^n)/\mathbb Z_m, F_m(\mathbb R^n)/\mathbb Z_m; \mathbb F_m)$ is trivial for $i<n$ and generated by one element $x$ for $i = n$. Now from the long exact sequence of the pair $(F_m(S^n)/\mathbb Z_m, F_m(\mathbb R^n)/\mathbb Z_m)$ we obtain that the group $H^n(F_m(S^n)/\mathbb Z_m; \mathbb F_m)$ is spanned by two elements: $w^{n/2}$ and the image of $x$, which we denote again by $x$. If this image is zero then we are done, so we may assume that $x\neq 0\in H^n(F_m(S^n)/\mathbb Z_m; \mathbb F_m)$. The multiplication is determined by the relations $x^2=0$ (from the cohomology of $(D^n, S^n)$) and $wx=0$, because the pairs $(F_m(S^n), X_i)$ are permuted freely by $\mathbb Z_m$ and the entire cohomology $H^*(F_m(S^n)/\mathbb Z_m, F_m(\mathbb R^n)/\mathbb Z_m; \mathbb F_m)$ is therefore annihilated by $w$. 

Now assume that $h_0''^*(\alpha) = ax + bw^{n/2}$. By taking squares we obtain $(ax+bw^{n/2})^2 = b^2w^n = 0$ and therefore $b = 0$. Hence we have to calculate $y = \iota^*(x)\in H^n(D_m(S^n)/\mathbb Z_m; \mathbb F_m)$. Consider the $(n-1)$-sphere $Y_1\subset D_m(S^n)$ formed by $m$-gons with the first vertex coinciding with $x_0$. It is clear that $y$ is Poincar\'e dual to the cycle 
$$
c_y = \left(\bigcup_{g\in \mathbb Z_m} g(Y_1)\right)/\mathbb Z_m.
$$ 
It is clear that the cycle $c_y$ of $D_m(S^n)/\mathbb Z_m$ is the natural image of the cycle $Y_1$ of $D_m(S^n)$. But the latter cycle represents the generator of $H_{n-1}(D_m(S^n); \mathbb Z)=\mathbb Z_2$. Hence modulo $m$ we have $[S_1]=0$, $[c_y]=0$, and therefore $y = 0$.

Now we have established that $h''^*(\alpha) = 0\in H^n(D_m(S^n); \mathbb F_m)$; the rest of the proof proceeds as in the previous arguments.

\section{Proof of Theorem~\ref{even-path}}
\label{even-sec}

Let us start with the case $m=4$. Consider the subspace $Q_4(S^n)\subset F_4(S^n)$ consisting of $4$-tuples $(x_1, x_2, x_3, x_4)$ satisfying the equations (the distance is measured in the ambient Euclidean $\mathbb R^{n+1}$):
\begin{eqnarray}
\label{q-eq1}
\|x_1 - x_3\| = \|x_2 - x_4\| &=& 2\delta\\
x_1+x_3 + x_2 + x_4 &=& 0\\
\label{q-eq3}
\|x_1+x_3-x_2-x_4\| &=& 4\sqrt{1-\delta^2}.
\end{eqnarray}
Here $\delta$ is a fixed number satisfying $0<\delta<1$. Informally, we consider configurations of four points such that the midpoints of the geodesic segments $[x_1,x_3]$ and $[x_2,x_4]$ are antipodal and the lengths of these segments are fixed. This subspace is obviously $D_8$-invariant. In the particular case $\delta=1/\sqrt{2}$ this space includes squares inscribed in $S^n$.

Consider another configuration space $Q_4(\mathbb R^{n+1})$ given by the same equations (\ref{q-eq1})--(\ref{q-eq3}) but without the assumption that $x_1,\ldots, x_4\in S^n$. If $\delta<1/\sqrt{2}$ we have the inclusion $Q_4(\mathbb R^{n+1})\subset F_4(\mathbb R^{n+1})$. Informally, $Q_4(\mathbb R^{n+1})$ consists of $4$-tuples $(x_1, x_2, x_3, x_4)$ such that the distances $\|x_1-x_3\|$ and $\|x_2-x_4\|$ are fixed, the midpoints of $[x_1,x_3]$ and $[x_2,x_4]$ are antipodal, and the distance between those midpoints is also fixed. Obviously $Q_4(\mathbb R^{n+1})$ can be identified with $S^n\times S^n\times S^n$ because the directions of $x_1+x_3-x_2-x_4$, $x_1-x_3$, and $x_2-x_4$ define uniquely the configuration $(x_1,x_2,x_3,x_4)\in Q_4(\mathbb R^{n+1})$ and can be prescribed arbitrarily. Thus $Q_4(\mathbb R^{n+1})$ is $(n-1)$-connected.

The space $Q_4(S^n)$ is not so simple as $Q_4(\mathbb R^{n+1})$. The map $\pi : Q_4(S^n)\to S^n$ defined by 
\begin{equation}
\label{pi-q-def}
\pi : (x_1,x_2,x_3,x_4)\mapsto \frac{x_1+x_3}{\|x_1+x_3\|}
\end{equation}
makes $Q_4(S^n)$ a bundle over $S^n$ with fiber $S^{n-1}\times S^{n-1}$. More precisely, if we denote by $\tau$ the tangent bundle of $S^n$ with its corresponding sphere bundle $S(\tau)$, then $Q_4(S^n) = S(\tau)\times_{S^n} S(\tau)$. Hence the space $Q_4(S^n)$ is $(n-2)$-connected.

If $n$ grows both spaces $Q_4(S^n)/D_8$ and $Q_4(\mathbb R^{n+1})/D_8$ approach the classifying space $BD_8$. In fact, $Q_4(\mathbb R^{n+1})/D_8$ is a standard approximation to $BD_8$ which is useful in the study of the cohomology $H^*(D_8; \mathbb F_2)$, see the definition of the wreath product of projective spaces ($\widetilde M(q,4)=Q_4(\mathbb R^{q+1})$ and $M(q,4)=Q_4(\mathbb R^{q+1})/D_8$ in our notation) in N.~Hung~\cite[Section~1, p.~254]{hung1990}. We need the following fact about the cohomology of $Q_4(S^n)/D_8$:

\begin{lem}
\label{coh-surj}
The natural map $\kappa_8^* : H^*(D_8;\mathbb F_2)\to H^*(Q_4(S^n)/D_8; \mathbb F_2)$ is surjective in dimensions $\le n$.
\end{lem}

Until the end of this section we assume mod $2$ cohomology and omit the coefficients $\mathbb F_2$ from the notation.

\begin{proof}
Note that $BD_8 = Q_4(\mathbb R^\infty)/D_8$ is a bundle over $\mathbb RP^\infty$ with fiber $\mathbb RP^\infty\times \mathbb RP^\infty$. Its corresponding spectral sequence starts with 
$$
E_2 = H^*(\mathbb Z_2; H^*(\mathbb Z_2)\otimes H^*(\mathbb Z_2) ),
$$
where we identify $H^*(\mathbb RP^\infty)$ with $H^*(\mathbb Z_2)$ and assume that $\mathbb Z_2$ acts on $H^*(\mathbb Z_2)\otimes H^*(\mathbb Z_2)$ by permuting the factors. A lemma of Nakaoka~\cite{naka1961} (see also Theorem~2.1 and the remark about the graded algebra structure after it in I.~Leary~\cite{lea1997}) asserts that this spectral sequence collapses at $E_2$ and there is an isomorphism of graded algebras
$$
H^*(D_8) = H^*(\mathbb Z_2\wr \mathbb Z_2) \cong H^*(\mathbb Z_2; H^*(\mathbb Z_2)\otimes H^*(\mathbb Z_2) ).
$$

Now we consider $Q_4(S^n)/D_8$ as a bundle over $\mathbb RP^n$ with fiber $\mathbb RP^{n-1}\times \mathbb RP^{n-1}$. The homomorphism of $E_2$-terms
$$
\kappa_8^* : H^*(\mathbb Z_2; H^*(\mathbb Z_2)\otimes H^*(\mathbb Z_2) ) \to H^*(\mathbb RP^n; H^*(\mathbb RP^{n-1})\otimes H^*(\mathbb RP^{n-1}) )
$$
is surjective in dimensions $\le n$. Indeed, it is an isomorphism on $E_2^{p,q}$ for $p+q\le n$ and $q\le n-1$, while in $E_2^{0,n}$ it annihilates the element $$
c_1^n\otimes 1+1\otimes c_2^n\in H^0(\mathbb Z_2, H^*(\mathbb RP^\infty)\otimes H^*(\mathbb RP^\infty)),
$$
where we denote the generators of $H^1(\mathbb RP^\infty)$ in the respective factors by $c_1$ and $c_2$. Therefore the resulting homomorphism
$$
\kappa_8^* : H^*(D_8) \to H^*(Q_4(S^n)/D_8)
$$
is surjective in dimensions $\le n$.
\end{proof}

Now we are going to use Lemma~\ref{reduction2}. The subgroup of $D_8$ that fixes $x_1$ is the copy of $\mathbb Z_2$ that exchanges $x_2$ and $x_4$. The map $\pi_1: Q_4(S^n)\to S^n$ of Lemma~\ref{reduction2} is $\mathbb Z_2$-equivariantly homotopic to the map $\pi$ defined in (\ref{pi-q-def}) by connecting $x_1$ and $\frac{x_1+x_3}{\|x_1+x_3\|}$. By Lemma~\ref{reduction2} we assume that $\pi$ is $\mathbb Z_2$-equivariantly homotopic to some $D_8$-equivariant map $h : Q_4(S^n)\to S^n$.

Denote the corresponding maps by $h_2 : Q_4(S^n)/\mathbb Z_2 \to S^n$, $h_8 : Q_4(S^n)/D_8\to S^n$, and $q: Q_4(S^n)/\mathbb Z_2\to Q_4(S^n)/D_8$ so that $h_2 = h_8\circ q$. From the commutative diagram
$$
\begin{CD}
H^n(Q_4(S^n)/\mathbb Z_2) @<{q^*}<< H^n(Q_4(S^n)/D_8)\\
@A{\kappa_2^*}AA @A{\kappa_8^*}AA\\
H^n(\mathbb Z_2) @<<< H^n(D_8)
\end{CD}
$$
and Lemma~\ref{coh-surj} we obtain that the image $q^*(H^n(Q_4(S^n)/D_8))$ (and therefore the image $h_2^*(H^n(S^n))$) is contained in the image $\kappa_2^*(H^n(\mathbb Z_2))$.

The space $Q_4(S^n)/\mathbb Z_2$ is a bundle $S(\tau)\times_{S^n} P(\tau)$ over $S^n$. Here $P(\tau)$ is the projective bundle corresponding to $\tau$ and the bundle projection is the map $\pi$ defined in (\ref{pi-q-def}). The spectral sequence of this bundle collapses and therefore 
\begin{equation}
\label{byZ2-coh}
H^*(Q_4(S^n)/\mathbb Z_2) \cong H^*(S^{n-1})\otimes H^*(\mathbb RP^{n-1})\otimes H^*(S^n)
\end{equation}
additively. Since all the Stiefel--Whitney classes of $\tau$ vanish we in fact obtain an isomorphism of graded algebras. It follows that the image of $\kappa_2^*$ is multiplicatively generated by $H^1(\mathbb RP^{n-1})$ and does not contain the nontrivial element of $\pi^*(H^n(S^n))$. Thus the conditions of Lemma~\ref{reduction2} cannot be satisfied. Another way to prove this without the multiplicative structure in (\ref{byZ2-coh}) is to note that there is an $(n-1)$-sphere bundle $\varphi : Q_4(S^n)/\mathbb Z_2\to P(\tau)$ and the maps $\kappa_2^*$ and $\pi^*$ pass through $H^*(P(\tau))$ (because the maps $\kappa_2$ and $\pi$ factor through $\varphi$). The vanishing of Stiefel--Whitney classes of $\tau$ ensures the isomorphism of graded algebras $H^*(P(\tau)) \cong H^*(\mathbb RP^{n-1})\otimes H^*(S^n)$; hence the image of $\kappa_2^*$ is generated by $H^1(\mathbb RP^{n-1})$ and cannot contain the generator of $\pi^*(H^n(S^n))$. 

Finally we treat the case of arbitrary even $m=2\ell > 4$ by producing a $D_8$-equivariant embedding $Q_4(S^n)\to F_m(S^n)$. This embedding is constructed by replacing the pair $(x_1,x_3)$ with $\ell$ points (including the end points) distributed uniformly on the geodesic segment $[x_1, x_3]$, and the same for the segment $[x_2,x_4]$. A $\Sigma_m$-invariant map $h_m : F_m(S^n)\to S^n$ therefore induces a $D_8$-invariant map $h_4 : Q_4(S^n)\to S^n$ which has already been shown to be impossible.

\section{Remarks about the case $m=3$ and $n=4$}
\label{m3-n4-sec}

In the open case $m=3$, $n=4$ it would be sufficient (following the approach in Section~\ref{tri-path-0mod4-sec}) to find a cohomology operation of positive degree that takes $w^2\in H^4(D_3(S^4)/\mathbb Z_3; \mathbb F_3)$ to a nonzero element of dimension $\le 7$. The Steenrod reduced powers fail to do so (see Section~\ref{tri-path-0mod4-sec}) and moreover squaring does not work (see Remark~\ref{F3Rn-rem}). 

\begin{rem}
The rest of this section replaces $\Sigma_3$ by its subgroup $\mathbb Z_3$, so we try to solve positively a relaxed version of Question~\ref{mult-path} without any success. 
\end{rem}

A relevant example is the original Hopf map $S^7\to S^4$ (considered as the natural projection of unit vectors in $\mathbb H^2$ to $\mathbb HP^1 = S^4$), which is $\mathbb Z_3$-equivariant (we may assume $\mathbb Z_3\subset \mathrm{Sp}(1)$) and maps the generator $\alpha\in H^4(S^4; \mathbb F_3)$ to a nonzero multiple of 
$$
w^2\in H^4(S^7/\mathbb Z_3; \mathbb F_3) = H^4(\mathbb Z_3; \mathbb F_3).
$$
So the class $w^2\in H^4(S^7/\mathbb Z_3; \mathbb F_3)$ is annihilated by any cohomology operation of positive degree. The sphere $S^7$ has the same mod $3$ cohomology (and $\mathbb Z_3$-equivariant cohomology mod $3$) as $D_3(S^4)$. We note one difference, 
%without any idea if it may help
but do not know if it will prove helpful: The action of $\mathbb Z_3$ on $S^7$ cannot be extended to an action of $\Sigma_3\supset \mathbb Z_3$.

In an attempt to give a positive answer for Question~\ref{mult-path} for $m=3$, $n=4$ we may try to construct a $\mathbb Z_3$-equivariant map $D_3(S^4)\to S^4$ homotopic to $\pi_1\circ \iota$ as in Lemma~\ref{reduction-tri}. In the case $n=2$ such a map $D_3(S^2)\to S^2$ is easily constructed by considering two ``centers'' of the triangle and selecting one of them according to the orientation (here $\mathbb Z_3$-equivariance is essential). 

But the situation is more difficult in the case $n=4$. For a triple $(a,b,c)\in D_3(S^4)$ denote its $2$-dimensional linear span by $\alpha(a,b,c)$ and its $3$-dimensional orthogonal complement by $\beta(a,b,c)$. It would be sufficient to find a $\mathbb Z_3$-equivariant map $h : D_3(S^4)\to S^4$ such that $h(a,b,c)\in \beta(a,b,c)$ for any triple. But G.~Whitehead showed in~\cite{wh1963} that we cannot assign continuously to any pair $(x_1, x_2)\in V_{n+1, 2}$ another vector $x_3$ orthogonal to $x_1$ and $x_2$ except for the cases $n=2,6$, even without assuming any invariance under a group action on $V_{n+1,2}$.

%RK The theorem is replaced by references to stronger results of Whitehead.

%The rest of Section 8, concerning the cohomology of $D_3(S^4)/D_6$ has been omitted (and has been moved below following the paper.

\end{document}